\documentclass[12pt]{article}
\usepackage{amscd, amssymb, amsmath, calc, enumerate}
\begin{document} 
\renewcommand{\thesubsection}{\arabic{subsection}}
\newenvironment{eq}{\begin{equation}}{\end{equation}}
\newenvironment{proof}{{\bf Proof}:}{\vskip 5mm }
\newenvironment{rem}{{\bf Remark}:}{\vskip 5mm }
\newenvironment{remarks}{{\bf Remarks}:\begin{enumerate}}{\end{enumerate}}
\newenvironment{examples}{{\bf Examples}:\begin{enumerate}}{\end{enumerate}}  
\newtheorem{proposition}{Proposition}[subsection]
\newtheorem{lemma}[proposition]{Lemma}
\newtheorem{definition}[proposition]{Definition}
\newtheorem{theorem}[proposition]{Theorem}
\newtheorem{cor}[proposition]{Corollary}
\newtheorem{conjecture}{Conjecture}
\newtheorem{pretheorem}[proposition]{Pretheorem}
\newtheorem{hypothesis}[proposition]{Hypothesis}
\newtheorem{example}[proposition]{Example}
\newtheorem{remark}[proposition]{Remark}
\newtheorem{ex}[proposition]{Exercise}
\newtheorem{cond}[proposition]{Conditions}
\newtheorem{cons}[proposition]{Construction}
\newcommand{\llabel}[1]{\label{#1}}
\newcommand{\comment}[1]{}
\newcommand{\sr}{\rightarrow}
\newcommand{\lr}{\longrightarrow}
\newcommand{\xr}{\xrightarrow}
\newcommand{\dw}{\downarrow}
\newcommand{\bdl}{\bar{\Delta}}
\newcommand{\zz}{{\bf Z\rm}}
\newcommand{\zq}{{\bf Z}_{qfh}}
\newcommand{\nn}{{\bf N\rm}}
\newcommand{\qq}{{\bf Q\rm}}
\newcommand{\nq}{{\bf N}_{qfh}}
\newcommand{\oo}{\otimes}
\newcommand{\uu}{\underline}
\newcommand{\ih}{\uu{Hom}}
\newcommand{\af}{{\bf A}^1}
\newcommand{\wt}{\widetilde}
\newcommand{\gm}{{\bf G}_m}
\newcommand{\dsr}{\stackrel{\sr}{\scriptstyle\sr}}
\newcommand{\PP}{$P_{\infty}$}
\newcommand{\tp}{\tilde{D}}
\newcommand{\HH}{$H_{\infty}$}
\newcommand{\ii}{\stackrel{\scriptstyle\sim}{\sr}}
\newcommand{\BB}{_{\bullet}}
\newcommand{\D}{\Delta}
\newcommand{\colim}{{\rm co}\hspace{-1mm}\lim}
\newcommand{\cf}{{\it cf} }
\newcommand{\msf}{\mathsf }
\newcommand{\mcal}{\mathcal }
\newcommand{\ep}{\epsilon}
\newcommand{\tl}{\widetilde}
\newcommand{\ub}{\mbox{\rotatebox{90}{$\in$}}}
\newcommand{\piece}{\vskip 3mm\noindent\refstepcounter{proposition}{\bf
\theproposition}\hspace{2mm}}
\newcommand{\subpiece}{\vskip 3mm\noindent\refstepcounter{equation}{\bf\theequation}\hspace{2mm}}{\vskip
3mm}
\numberwithin{equation}{subsection}
%
%
\begin{center}
{\Large\bf Motives over simplicial schemes}\\
\vskip 4mm
{\large Vladimir Voevodsky}\\
{\em Preliminary version, June 16 2003}
\end{center}
\vskip 4mm
\tableofcontents
\subsection{Introduction}
This paper was written as a part of \cite{zl} and is intended
primarily to provide the definitions and results about
motives over simplicial schemes used in the proof of the Bloch-Kato
conjecture. 

For the purpose of this paper {\em a scheme} means a disjoint union of
possibly infinitely many separated noetherian schemes of finite
dimension. A smooth scheme over a scheme $S$ is a disjoint union of
smooth separated schemes of finite type over $S$. A smooth simplicial
scheme $\cal X$ over $S$ is a simplicial scheme such that all terms of
$\cal X$ are smooth schemes over $S$ and all morphisms are over $S$.

If $\cal X$ is a smooth simplicial scheme over a field $k$ then the
complex of presheaves with transfers defined by the simplicial
presheaf with transfers $\zz_{tr}({\cal X})$ gives an object
$M({\cal X})$ in the triangulated category of motives
$DM_{-}^{eff}(k)$ over $k$. The motivic cohomology of this object are
called the motivic cohomology of $\cal X$ and we denote these groups
by
$$H^{p,q}({\cal X},A):=Hom_{DM}(M({\cal X}),A(q)[p])$$
where $A$ is an abelian group of coefficients. 

The main goal of this paper is to define for any smooth simplicial
scheme ${\cal X}$ over a perfect field $k$ a tensor triangulated
category $DM^{eff}_{-}({\cal X})$ such that
\begin{eq}
\llabel{formulaneed}
H^{p,q}({\cal X},A)=Hom_{DM^{eff}({\cal X})}(\zz,A(p)[q]).
\end{eq}
For completeness we give our construction of $DM^{eff}_{-}({\cal X})$
in the case of a general simplicial scheme and in particular we
provide a definition for ``motivic cohomology'' of simplicial schemes
based on (\ref{formulaneed}). If the terms of $\cal X$ are not regular
there are examples which show that the motivic cohomology defined by
(\ref{formulaneed}) do not satisfy the suspension isomorphism with
respect to the $T$-suspension (which implies that they do not satisfy
the projective bundle formula and do not have the Gysin long exact
sequence). Therefore in the general case we have to distinguish the
``effective'' motivic cohomology groups given by (\ref{formulaneed})
and the stable motivic cohomology groups given by
\begin{eq}
\llabel{stablegroups}
H^{p,q}_{stable}({\cal X},A):=\lim_{n} Hom_{DM^{eff}({\cal
X})}(\zz(n),A(n+q)[p])
\end{eq}
The stable motivic cohomology groups should also have a descrption as
morphisms bewteen the Tate objects in the properly defined $T$-stable
version of $DM$ and should have  many good properties including the
long exact sequence for blow-ups which the unstable groups in the
non-regular case do not have. 

If the terms of $\cal X$ are regular schemes of equal characteristic
the cancellation theorem over perfect fields implies that this problem
does not arise and the stable groups are same as the effective ones
(see Corollary \ref{cancell}). Since in applications to the Bloch-Kato
conjecture we need only the case of smooth schemes over a perfect
field we do not consider stable motivic cohomology in this paper.

Note also that while we use schemes smooth over a base as the basic
building blocks of motives over this base one can also consider all
(separated) schemes instead as it is done in \cite{cancellation} and
\cite{zslice}. As far as the constructions of this paper are concerned this
make no difference except that the resulting motivic category gets
bigger.
%

\subsection{Presheaves with transfers}
\llabel{ss1}
Let ${\cal X}$ be a simplicial scheme with terms $X_i$, $i\ge 0$. For a
morphism $\phi:[j]\sr [i]$ in $\Delta$ we let ${\cal X}_{\phi}$ denote the
corresponding morphism ${\cal X}_i\sr {\cal X}_j$. Denote by $Sm/{\cal X}$ the category
defined as follows:
\begin{enumerate}
\item objects of $Sm/{\cal X}$ are pairs of the form $(Y,i)$
where $i$ is a non-negative integer and $Y\sr {\cal X}_i$ is a smooth
scheme over ${\cal X}_i$
\item a morphism from $(Y,i)$ to $(Z,j)$ is a pair
$(u,\phi)$ where $\phi:[j]\sr [i]$ is a morphism in $\Delta$ and
$u:Y\sr Z$ is a morphism of schemes such that the square
\begin{eq}
\llabel{starptwo}
\begin{CD}
Y @>u>> Z\\
@VVV @VVV\\
{\cal X}_i @>{\cal X}_{\phi}>> {\cal X}_j
\end{CD}
\end{eq}
commutes. 
\end{enumerate}
A presheaf of sets on $Sm/{\cal X}$ is a contravariant functor from
$Sm/{\cal X}$ to sets. Each presheaf $F$ on $Sm/{\cal X}$ defines in the
obvious way a famlily of presheaves $F_i$ on $Sm/{\cal X}_i$ together with
natural transformations $F_{\phi}:{\cal X}_{\phi}^*(F_j)\sr F_i$ given for
all morphisms $\phi:[j]\sr [i]$ in $\Delta$. 

One can easily see that this construction provides a bijection between
presheaves on $Sm/{\cal X}$ and families $(F_i,F_{\phi})$ such that
$F_{Id}=Id$ and the obvious compatibility condition holds for
composable pairs of morphisms in $\Delta$. Under this bijection the
presheaf $h_{(Y,i)}$ represented by $(Y,i)$ has as its $j$-th
component the presheaf 
$$(h_{(Y,i)})_j=\coprod_\phi h_{Y\times_{\phi} {\cal X}_j}$$
where $\phi$ runs through the morphisms $[i]\sr [j]$ in $\Delta$.

Our first goal is to develop an analog of this picture where the
presheaves of sets are replaced with presheaves with transfers. Let us
recall first the basic notions for the presheaves with transfers over
usual schemes. For a scheme $X$ denote by $SmCor(X)$ the category
whose objects are smooth schemes over $X$ and morphisms are finite
correspondences over $X$ (in the case of a non-smooth $\cal X$ see
\cite{cancellation} for a detailed definition of finite
correspondences and their compositions). Note that we allow schemes
which are infinite disjoint unions of smooth schemes of finite type to
be objects of $SmCor(X)$. In particular our $SmCor(X)$ has infinite
direct sums. A presheaf with transfers on $Sm/X$ is an additive
contravariant functor from $SmCor(X)$ to abelian groups which takes
infinite direct sums to products. Presheaves with transfers form an
abelian category $PST(X)$. The forgetful functor from presheaves with
transfers to presheaves of sets has a left adjoint which we denote by
$\zz_{tr}(-)$. If $Y$ is a smooth scheme over $X$ and $h_{Y}$ is the
presheaf of sets represented by $Y$ then $\zz_{tr}(h_Y)$ coincides
with the presheaf with transfers represented by $Y$ on $SmCor(X)$ and
we denote this object by $\zz_{tr}(Y)$. It will be convenien for us to
identify $SmCor(X)$ with its image in $PST(X)$ and denote the object
of $SmCor(X)$ corresponding to a smooth scheme $Y$ over $X$ by
$\zz_{tr}(Y)$.

A morphism of schemes $f:Y\sr X$ defines the pull-back functor
$$\zz_{tr}(U)\mapsto \zz_{tr}(U\times_X Y)$$
from $SmCor(X)$ to $SmCor(Y)$ and therefore a pair of adjoint functors
$f_*, f^*$ between the corresponding categories of presheaves with
transfers. Since $f_*$ commutes with the forgetful functor we conclude
by adjunction that for a presheaf of sets $F$ over $X$ one has
\begin{eq}
\llabel{com1}
\zz_{tr}(f^*(F))=f^*(\zz_{tr}(F)).
\end{eq}
Note that it is not necessarily true that the pull-back functors on
the presheaves of sets and the presheaves with transfers commute with
the forgetful functor.
\begin{definition}
\llabel{pt}
Let ${\cal X}$ be a simplicial scheme. A presheaf with transfers on
${\cal X}$ is the following collection of data:
\begin{enumerate}
\item For each $i\ge 0$ a presheaf with transfers $F_i$ on
$Sm/{\cal X}_i$ 
\item For each morphism $\phi:[j]\sr [i]$ in the simplicial category $\Delta$
a morphism of presheaves with transfers
$$F_{\phi}:{\cal X}_{\phi}^*(F_j)\sr F_i$$
\end{enumerate}
These data should satisfy the condition that $F_{id}=Id$ and for a
composable pair of morphisms $\phi:[j]\sr [i]$, $\psi:[k]\sr [j]$ in
$\Delta$ the obvious diagram of morphisms of presheaves commutes.
\end{definition}
We let $PST({\cal X})$ denote the category of presheaves with
transfers on ${\cal X}$. This is an abelian category with kernels
and cokernels computed termise.
\begin{example}\rm
Let $X$ be a scheme and ${\cal X}$ be such that ${\cal X}_i=X$ for all $i$ and
all the structure morphisms are identities. Then a presheaf with
transfers over ${\cal X}$ is the same as a cosimplicial object in the
category of presheaves with transfers over $X$.
\end{example}
Let $F=(F_i,F_{\phi})$ be a presheaf of sets on $Sm/{\cal X}$. In view
of (\ref{com1}), the collection of presheaves with transfers
$\zz_{tr}(F_i)$ has a natural structure of a presheaf with transfers
on $Sm/{\cal X}$ which we denote by $\zz_{tr}(F)$. One observes easily
that $F\mapsto \zz_{tr}(F)$ is the left adjoint to the corresponding
forgetful functor. If $(Y,i)$ is an object of $Sm/{\cal X}$ and
$h_{(Y,i)}$ is the corresponding representable presheaf of sets we let
$\zz_{tr}(Y,i)$ denote the presheaf with transfers
$\zz_{tr}(Y,i)$. For any presheaf with transfers $F$ we have
\begin{eq}
\llabel{equal1}
Hom(\zz_{tr}(Y,i),F)=F_i(Y).
\end{eq}
By construction, the $i$-th component of $\zz_{tr}(Z,j)$ is 
\begin{eq}
\llabel{need6}
\zz_{tr}((h_{(Z,j)})_i)=\zz_{tr}(\coprod_\phi h_{Z\times_{\phi}
{\cal X}_i})
=\oplus_\phi \zz_{tr}(Z\times_{\phi}
{\cal X}_i)
\end{eq}
where $\phi$ runs through all morphisms $[j]\sr [i]$ in
$\Delta$. Together with (\ref{equal1}) this shows that
$$Hom(\zz_{tr}(Y,i),\zz_{tr}(Z,j))=\oplus_{\phi} Hom_{SmCor({\cal
X}_i)}(Y, Z\times_{\phi} {\cal X}_i)$$

Denote by $SmCor({\cal X})$ the full subcategory in $PST({\cal X})$
generated by direct sums of objects of the form $\zz_{tr}(Y_i)$. The
following lemma is an immediate corollary of 
(\ref{equal1}).
\begin{lemma}
\llabel{ident} The category $PST({\cal X})$ is naturally equivalent to
the category of additive contravariant functors from $SmCor({\cal X})$
to the category of abelian groups which commute with $\oplus$.
\end{lemma}
Lemma \ref{ident} implies in particular that we can apply in the
context of $PST({\cal X})$ the usual construction of a canonical left
resolution of a functor by direct sums of representable functors. It
provides us with a functor $Lres$ from $PST({\cal X})$ to complexes over
$SmCor({\cal X})$ together with a familiy of natural quasi-isomorphisms 
$$Lres(F)\sr F$$
We let 
$$D({\cal X}):=D_{-}(PST({\cal X}))$$
denote be the derived category of complexes bounded from the above
over $PST({\cal X})$. In view of Lemma \ref{ident} it can be
identified with the homotopy category of complexes bounded from the
above over $SmCor({\cal X})$ by means of the functor
\begin{eq}
\llabel{lres}
K\mapsto Tot(Lres(K))
\end{eq}
which we also denote by $Lres$. 

For a morphism of simplicial schemes $f_{\bullet}:{\cal X}\sr {\cal
Y}$ the direct and inverse image functors $f_i^*, f_{i,*}$ define in
the obvious way functors
$$f_{\bullet}^*:PST({\cal Y})\sr PST({\cal X})$$
$$f_{\bullet,*}:PST({\cal X})\sr PST({\cal Y})$$
and the adjunction morphisms $Id\sr f_{i,*}f^*_i$, $f^*_if_{*,i}\sr Id$
define morphisms 
$$Id\sr f_{\bullet,*}f_{\bullet}^*$$
$$f_{\bullet,*}f_{\bullet}^*\sr Id$$
which automatically satisfy the adjunction axioms and therefore make
$f_{\bullet,*}$ into a right adjoint to $f^*_{\bullet}$. 

The functor $f^*\BB$ takes $\zz_{tr}(Z,i)$ to
$\zz_{tr}(Z\times_{{\cal Y}_i}{\cal X}_i,i)$ and commutes with direct
sums. Therefore it restricts to a functor 
$$f^{-1}_{\bullet}:SmCor({\cal Y})\sr SmCor({\cal X}).$$
Using the equivalence of Lemma \ref{ident} we can now recover the
functors $f^*\BB$ and $f_{*,\bullet}$ as the direct and inverse image functors
defined by $f^{-1}\BB$.

The functors $f_{\bullet,*}$ are clearly
exact and therefore define functors on the corresponding derived
categories. The functors $f_{\bullet}^*$ for non-smooth $f$ are in
general only right exact but not left exact. To define the
corresponding left adjoints one sets
$$Lf_{\bullet}^*(K)=f_{\bullet}^*(Lres(K))$$
where $Lres$ is defined on complexes by (\ref{lres}).  The
corresponding functor on the derived categories, which we continue to
denote by $Lf^*_{\bullet}$ is then a left adjoint to $f_{\bullet,*}$.

A group of functors relates the presheaves with transfers over
$\cal X$ with the presheaves with transfers over the terms of $\cal
X$. For any $i\ge 0$ let
$$r_i:SmCor({\cal X}_i)\sr SmCor({\cal X})$$
be the functor which takes a smooth scheme $Y$ over ${\cal X}_i$ to
$\zz_{tr}(Y,i)$. This functor defines in the usual way a pair of
adjoints 
$$r_{i,\#}:PST({\cal X}_i)\sr PST({\cal X})$$
and
$$r_i^*:PST({\cal X})\sr PST({\cal X}_i)$$
where $r_i^*$ is the right adjoint and $r_{i,\#}$ the left
adjoint. Equation (\ref{equal1}) implies that for a presheaf with
transfers $F$ on $\cal X$, $r_i^*(F)$ is the $i$-th component of
$F$. To compute $r_{i,\#}$ note that 
\begin{eq}
\llabel{foresh1}
r_{i,\#}(\zz_{tr}(Y))=\zz_{tr}(Y,i)
\end{eq}
and $r_{i,\#}$ is right exact. Therefore for a presheaf with transfers
$F$ over ${\cal X}_i$ one has
\begin{eq}
\llabel{foresh2}
r_{i,\#}(F)=h_0(r_{i,\#}(Lres(F)))
\end{eq}
where $Lres$ is the canonical left resolution by representable
presheaves with transfers and the right hand side of (\ref{foresh2})
is defined by (\ref{foresh1}). 

The functors $r_{i}^*$ are exact and therefore define functors
between the corresponding derived categories which we again denote by
$r_i^*$. We do not know if the functors $r_{i,\#}$ are exact but in any
event one can define the left derived functor
$$Lr_{i,\#}:=r_{i,\#}\circ Lres$$
This functor respects quasi-isomorphisms and the corresponding functor
between the derived categories which we continue to denote by
$Lr_{i,\#}$ is the left adjoint to $r_{i}^*$. 
\begin{lemma}
\llabel{isconserv}
The family of functors
$$r_i^*:D({\cal X})\sr D({\cal X}_i)$$
is conservative i.e. if $r_i^*(K)\cong 0$ for all $i$ then $K\cong 0$.
\end{lemma}
\begin{proof}
Let $K$ be an object such that $r_i^*(K)\cong 0$ for all $i$. Then by
adjunction
$$Hom(\zz_{tr}(Y,i),K[n])=Hom(Lr_{i,\#}(\zz_{tr}(Y)),K[n])=$$
$$=Hom(\zz_{tr}(Y),r_i^*(K)[n])=0.$$
Since objects of the form $\zz_{tr}(Y,i)$ generate $D({\cal X})$ we
conclude that $K\cong 0$.
\end{proof}
Consider the composition 
$$r_{i}^*r_{j,\#}:PST({\cal X}_j)\sr PST({\cal X}_i).$$
By (\ref{need6}) it takes $\zz_{tr}(Y)$ to $\oplus_{\phi}
\zz_{tr}(Y\times_{\phi} {\cal X}_i)$ where $\phi$ runs through
morphisms $[j]\sr [i]$ in $\Delta$. Therefore we have
\begin{eq}
\llabel{rirj}
r_{i}^*Lr_{j,\#}=\oplus_{\phi} L{\cal X}_{\phi}^*
\end{eq}
and passing to $h_0(-)$ we get
$$r_{i}^*r_{j,\#}=\oplus_{\phi} {\cal X}_{\phi}^*.$$
\begin{remark}
\rm\llabel{rmcover}
The functors $r_{i}$ behave as if the terms ${\cal X}_i$ formed a
covering of the simplicial scheme ${\cal X}$ with $r_i^*$ being the
inverse image functors for this covering and $r_{i,\#}$ being the
functors which in the case of an open covering $j_i:U_i\sr X$ are
denoted by $(j_i)_!$. 
\end{remark}
The functors $r_i^*$ commute in the obvious sense with the functors
$f^*$ for morphisms $f:{\cal X}\sr {\cal Y}$ of simplicial schemes.

Let now ${\cal X}$ be a simplicial scheme over a scheme $S$. We
have a functor
$$c^*:PST(S)\sr PST({\cal X})$$
which sends a presheaf with transfers $F$ over $S$ to the collection 
$$c^*(F)=(({\cal X}_i\sr S)^*(F))_{i\ge 0}$$
with the obvious structure morphisms. This functor is clearly right
exact and using the representable resolution $Lres$ over $S$ we may
define a functor $Lc^*$ from complexes over $PST(S)$ to complexes over
$PST({\cal X})$. Then $Lc^*$ respects quasi-isomorphisms and therefore
defines a triangulated functor
$$Lc^*:D(S)\sr D({\cal X})$$
The functors $c^*$ are compatible with the pull-back functors $f^*$
such that for $f:{\cal X}\sr {\cal Y}$ we have a natural isomorphism
$$c^*=f^*c^*$$
and, for the functors on the derived categories, we have natural
isomorphisms 
$$Lc^*=Lf^*Lc^*$$
They are also compatible with the functors $r_i^*$ such that one has
$$r_i^*c^*=p_i^*$$
and
$$r_i^*Lc^*=Lp_i^*$$
where $p_i$ is the morphism ${\cal X}_i\sr S$.

If $\cal X$ is a smooth simplicial scheme over $S$ then the functor
$c^*$ has a left adjoint $c_{\#}$ which takes $\zz_{tr}(Y,i)$
to the presheaf with transfers $\zz_{tr}(Y/S)$ on $SmCor(S)$. In
particular in this case $c^*$ is exact. The functor $c_{\#}$ being a
left adjoint is right exact and we use representable resolutions to
define the left derived
$$Lc_{\#}:=c_{\#}\circ Lres$$
The functor $Lc_{\#}$ respects quasi-isomorphisms and the
corresponding functor on the derived categories is a left adjoint to
$c^*=Lc^*$.

Functors $c_{\#}$ are compatible with the functors $r_{i,\#}$ such
that one has
$$r_{i,\#}c_{\#}=p_{i,\#}$$
where $p_i$ is the smooth morphism ${\cal X}_i\sr S$, and on the level
of the derived categories one has
$$Lr_{i,\#}Lc_{\#}=Lp_{i,\#}.$$

\subsection{Tensor structure}
\label{ss2}
Recall that for a scheme $X$ one uses the fiber product of smooth
schemes over $X$ and the corresponding external product of finite
correspondences to define the tensor structure on $SmCor(X)$. One then
defines a tensor structure on $PST(X)$ setting
$$F\oo G:=h_0(Lres(F)\oo Lres(G))$$
where the tensor product on the right is defined by the tensor product
on $SmCor(X)$. If $f:X'\sr X$ is a morphism of schemes then there are
natural isomorphisms
\begin{eq}
\llabel{oopb0}
f^*(F\oo G)=f^*(F)\oo f^*(G)
\end{eq}
which are compatible on representable presheaves with transfers with
the isomorphisms
$$(Y\times_X X')\times_{X'}(Z\times_X X')=(Y\times_X Z)\times_{X} X'$$
Let now ${\cal X}$ be a simplicial scheme. For presheaves with transfers
$F$, $G$ over ${\cal X}$ the collection of presheaves with transfers
$F_i\oo G_i$ over ${\cal X}_i$ has a natural structure of a presheaf with
transfers over ${\cal X}$ defined by isomorphisms (\ref{oopb0}). This
structure is natural in $F$ and $G$ and one can easily see that the
pairing 
$$(F,G)\mapsto F\oo G$$
extends to a tensor structure on presheaves with transfers over ${\cal
X}$. The unit of this tensor structure is the constant presheaf with
transfers $\zz$ which has as its components the constant presheaves
with transfers over ${\cal X}_i$.
The following lemma is straightforward.
\begin{lemma}
\llabel{froo}
Let $F$, $G$ be presheaves of sets over ${\cal X}$. Then there is a
natural isomorphism
$$\zz_{tr}(F\times G)=\zz_{tr}(F)\oo\zz_{tr}(G)$$
\end{lemma}
A major difference between the categories
of presheaves with transfers over a scheme and over a simplicial
scheme lies in the fact that the tensor structure on $PST({\cal X})$ does
not come from a tensor structure on $SmCor({\cal X})$. In particular, for
a general ${\cal X}$, $\zz$ is not representable and the tensor product of
two representable presheaves with transfers is not representable.

Let us say that a presheaf with transfers $F$ is admissible if its
components $F_i$ are direct sums of representable presheaves with
transfers over ${\cal X}_i$. The class of admissible presheaves contains
$\zz$ and is closed under tensor products. The following
straightforward lemma implies that any representable presheaf with
transfers is admissible and in particular that $Lres$ provides a
resolution by admissible presheaves.
\begin{lemma}
\llabel{repad}
A presheaf with transfers of the form $\zz_{tr}(Y,i)$ is admissible.
\end{lemma}
\begin{proof}
Follows immediately from (\ref{need6}).
\end{proof}
\begin{lemma}
\llabel{ooqi} Let $K, K', L$ be complexes of admissible presheaves
with transfers and $K\sr K'$ be a quasi-isomorphism. Then $K\oo L\sr
K'\oo L$ is a quasi-isomorphism.
\end{lemma}
\begin{proof}
The analog of this proposition for presheaves with transfers over each
${\cal X}_i$ holds since free presheaves with transfers are projective
objects in $PST({\cal X}_i)$. Since both quasi-isomorphisms and tensor
products in $PST({\cal X})$ are defined term-wise the proposition follows.
\end{proof}
In view of Lemmas \ref{repad} and \ref{ooqi} the functor
$$K\stackrel{L}{\oo}L:=Lres(K)\oo Lres(L)$$
respect quasi-isomorphisms in $K$ and $L$ and therefore defines a
functor on the derived categories which we also denote by
$\stackrel{L}{\oo}$. 

To see that this functor is a part of a good tensor triangulated
structure on $D({\cal X})$ we may use the following equivalent
definition. Let $A$ be the additive category of admissible presheaves
with transfers over ${\cal X}$ and $H_{-}(A)$ the homotopy category of
complexes bounded from the above over $A$. The tensor product of
presheaves with transfers makes $A$ into a tensor additive category
and we may consider the corresponding structure of the tensor
triangulated category on $H_{-}(A)$. Observe now that the natural
functor 
$$H_{-}(A)\sr D({\cal X})$$
is the localization with respect to the class of quasi-isomorphisms
and that the tensor product $\stackrel{L}{\oo}$ on $D({\cal X})$ is the
localization of the tensor product on $H_{-}(A)$. Since a tensor
trinagulated structure localizes well we conclude that $D({\cal X})$ is a
tensor trinagulated category with respect to $\stackrel{L}{\oo}$. More
precisely we have the following result.
\begin{proposition}
\llabel{mayoo0} The category $D({\cal X})$ is symmetric monoidal with
respect to the tensor product introduced above and this symmetric
monoidal structure satisfy axioms (TC1)-(TC3) of \cite{MayTT} with
respct to the standard triangulated structure.
\end{proposition}
The interaction between the tensor structure and the standard functors
introduced above are given by the following lemmas.
\begin{lemma}
\llabel{oopb}
For a morphism of simplicial schemes $f:{\cal X}\sr {\cal Y}$ one has
canonical isomorphisms in $PST({\cal X})$ of the form
\begin{eq}\llabel{ooup1}
f^*(F\oo G)=f^*(F)\oo f^*(G)
\end{eq}
and canonical isomorphisms in $D({\cal X})$ of the form
\begin{eq}\llabel{ooup2}
Lf^*(K\stackrel{L}{\oo} L)=Lf^*(K)\stackrel{L}{\oo}Lf^*(L).
\end{eq}
\end{lemma}
\begin{proof}
The first statement follows immediately from (\ref{oopb0}). The second
follows from the first and the fact that $f^*$ takes admissible objects to
admissible objects.
\end{proof}
\begin{lemma}
\llabel{rioo}
For a simplicial scheme ${\cal X}$ one has 
canonical isomorphisms in $PST({\cal X}_{i})$ of the form
$$r_i^*(F\oo G)=r_i^*(F)\oo r_i^*(G)$$
and canonical isomorphisms in $D({\cal X}_i)$ of the form
$$Lr_i^*(K\stackrel{L}{\oo} L)=Lr_i^*(K)\stackrel{L}{\oo}Lr_i^*(L).$$
\end{lemma}
\begin{lemma}
\llabel{ooc}
For a simplicial scheme ${\cal X}$ over a scheme $S$ one has
canonical isomorphisms in $PST({\cal X})$ of the form
$$c^*(F\oo G)=c^*(F)\oo c^*(G)$$
and canonical isomorphisms in $D({\cal X})$ of the form
$$Lc^*(K\stackrel{L}{\oo} L)=Lc^*(K)\stackrel{L}{\oo}Lc^*(L).$$
\end{lemma}
\begin{proof}
The first statement follows immediately from (\ref{oopb0}). The second
from the first and the fact that $f^*$ takes representable presheaves
with transfers over $S$ to admissible presheaves with transfers over
${\cal X}$.
\end{proof}
\begin{lemma}
\llabel{projform0}
For a simplicial scheme ${\cal X}$ over a scheme $S$ such that all ${\cal X}_i$
are smooth over $S$ one has
canonical isomorphisms in $PST({\cal X})$ of the form
$$c_{\#}(F\oo c^*(G))=c_{\#}(F)\oo G$$
and canonical isomorphisms in $D({\cal X})$ of the form
$$Lc_{\#}(F\stackrel{L}{\oo} Lc^*(G))=Lc_{\#}(F)\stackrel{L}{\oo}
G.$$
\end{lemma}
\begin{proof}
Since (\ref{ooup1}) holds and $c_{\#}$ is the left adjoint to $c^*$
there is a natural map
$$c_{\#}(F\oo c^*(G))\sr c_{\#}(F)\oo G.$$
Since all the functors here are right exact and every presheaf with
transfers is the colimit of a diagram of representable presheaves with
transfers it is sufficient to check that this map is an isomorphism
for representable $F$ and $G$. This follows immediately from the
isomorphisms
\begin{eq}\llabel{need5}
\zz_{tr}(Y,i)\oo c^*(\zz_{tr}(Z))=\zz_{tr}(Y\times_S Z,i)
\end{eq}
and
$$c_{\#}(\zz_{tr}(Y,i))=\zz_{tr}(Y/S)$$
The isomorphism (\ref{need5}) implies also that for a representable
$G$ and a representable $F$, $F\oo c^*(G)$ is representable. Therefore
the first statement of the lemma implies the second.
\end{proof}
To compute $Lc_{\#}$ on the constant sheaf we need the following
result.
\begin{lemma}
\llabel{res} Conisder the simplicial
object $L\zz_{\bullet}$ in $SmCor({\cal X})$ with terms 
$$L\zz_{i}=\zz_{tr}({\cal X}_i,i)$$
and the obvious structure morphisms. Let $L\zz_{*}$ be the
corresponding complex. Then there is a natural
quasi-isomorphism 
$$L\zz_{*}\sr \zz$$
\end{lemma}
\begin{proof}
We have to show that for any $(Y,j)$ the simplicial abelian
group $L\zz_{\bullet}(Y,j)$ is a resolution for the abelian
group 
$$\zz(Y)=H^0(Y,\zz)$$
Indeed one verifies easily that
$$L\zz_{\bullet}(Y,j)=\zz(\Delta^j)\oo \zz(Y)$$
and since $\Delta^j$ is contractible the projection $\Delta^j\sr pt$
defines a natural quasi-isomorphism $L\zz_{\bullet}(Y,j)\sr \zz(Y)$.
\end{proof}
If ${\cal X}$ is such that all its terms are disjoint unions of
smooth schemes over $S$ then we may consider the complex
$\zz_{tr}({\cal X})_*$ defined by the simplicial object represented by
${\cal X}$ in the derived categories of presheaves with transfers
over $S$. Note that 
$$\zz_{tr}({\cal X})=c_{\#}(L\zz\BB)$$
and therefore, Lemmas \ref{projform0} and \ref{res} impliy 
the following formula.
\begin{proposition}
\llabel{projform}
For a complex of presheaves with transfers $K$ over $S$ one has
$$Lc_{\#}Lc^*(K)\cong \zz_{tr}({\cal X})_*\stackrel{L}{\oo} K$$
\end{proposition}
%
%
\begin{remark}\rm
It is easy to see that the functors $Lf^*$ can be computed using more
general admissible resolutions instead of the representable
resolutions. But we can not use admissible resolitions to compute
$Lc_{\#}$ since for the (admissible) constant presheaf with transfers
$$\zz=Lc^*(\zz)$$ 
we have by \ref{projform}:
$$c_{\#}(\zz)=h_0(Lc_{\#}(\zz))=h_0(\zz_{tr}({\cal X})_*)\ne
\zz_{tr}({\cal X}_*)=Lc_{\#}(\zz)$$
\end{remark}
\begin{remark}\rm
It would be interesting to find a nice explicit description of the
complex $\zz_{tr}({\cal X}_i,i)\oo \zz_{tr}({\cal X}_j,j)$ or,
equivalently, a nice simplicial resolution of $h_{({\cal X}_i,i)}\times h_{({\cal
X}_j,j)}$ by representable presheaves (of sets).
\end{remark}

\subsection{Relative motives}
\llabel{ss3}
For a scheme $X$ let $W^{el}(X)$ be the class of complexes over
$PST(X)$ defined as follows:
\begin{enumerate}
\item for any (upper) distinguished square 
\begin{eq}
\llabel{updist}
\begin{CD}
W @>>> V\\
@VVV @VVV\\
U @>>> Y
\end{CD}
\end{eq}
in $Sm/X$, the corresponding Mayer-Vietoris complex 
$$\zz_{tr}(W)\sr \zz_{tr}(U)\oplus \zz_{tr}(V)\sr
\zz_{tr}(Y)$$
is in $W^{el}(X)$
\item for any $Y$ in $Sm/X$, the complex $\zz_{tr}(Y\times\af)\sr
\zz_{tr}(Y)$ is in $W^{el}(X)$.
\end{enumerate}
Let further $W(X)$ be the smallest class in $D(X)$ which contains
$W^{el}(X)$ and is closed under triangles, direct sums and direct
summands. One says that a morphism in $D(X)$ is an $\af$-equivalence
if its cone lies in $W(X)$ and defines the triangulated category
$DM_{-}^{eff}(X)$ of (effective, connective) motives over $X$ as
the localization of $D(X)$ with respect to $\af$-equivalences.

For a simplicial ${\cal X}$ consider 
$$W_{i}^{el}({\cal X}):=r_{i,\#}(W^{el}({\cal X}_i))$$
as classes of complexes in $PST({\cal X})$.  Let $W({\cal X})$ be the
smallest class in $D({\cal X})$ which contains all $W_i^{el}({\cal
X})$ and is closed under triangles, direct sums and direct summands.
\begin{definition}
\llabel{a1eq} A morphism $u$ in $D({\cal X})$ is called
an $\af$-equivalence if its cone lies in $W({\cal X})$.
\end{definition}
\begin{definition}
Let $\cal X$ be a simplicial scheme. The triangulated category
$DM_{-}^{eff}({\cal X})$ of (effective, connective) motives over $\cal
X$ is the localization of $D({\cal X})$ with respect to
$\af$-equivalences.
\end{definition}
\begin{lemma}
\llabel{resp}
\begin{enumerate}
\item For any morphism $f$ of simplicial schemes the functor $Lf^*$ takes
$\af$-equivalences to $\af$-equivalences,
\item for any simplicial scheme the functors $r_i^*$ take
$\af$-equivalences to $\af$-equivalences,
\item for any simplicial scheme the functors $Lr_{i,\#}$ take
$\af$-equivalences to $\af$-equivalences,
\item for any simplicial scheme over $S$ the functor $Lc^*$ takes
$\af$-equivalences to $\af$-equivalences,
\item for any smooth simplicial scheme over $S$ the functor $Lc_{\#}$
takes $\af$-equivalences to $\af$-equivalences.
\end{enumerate}
\end{lemma}
\begin{proof}
It follows immediately from the definitions that the functors $Lf^*$
and $Lc_{\#}$ and $r_{i,\#}$ take $\af$-equivalences to
$\af$-equivalences.  

The functor $r_{i}^*$ takes $\af$-equivalences to $\af$-equivalences
by (\ref{rirj}).

To see that $Lc^*$ takes $\af$-equivalences to
$\af$-equivalences consider a complex $L$ over $S$ which consists of
representable presheaves with transfers. Let further $L_i$ be the
pull-back of $L$ to ${\cal X}_i$ which we consider as a complex of
representable presheaves with transfers over ${\cal X}$. One has
$$Lc^*(L)=\zz\oo Lc^*(L)=L\zz_{*}\oo
Lc^*(L)$$
where $L\zz_*$ is the complex of Lemma \ref{res}. By (\ref{need5}) we
conclude that $Lc^*(L)$ is quasi-isomorphic to the total complex of a
bicomplex with terms of the form $r_{i,\#}(L_i)$. Since for $L\in
W^{el}(S)$ we have $L_i\in W^{el}({\cal X}_i)$ for all $i$ this
implies that $Lc^*$ takes $W(S)$ to $W({\cal X})$.
\end{proof}
We keep the notations $Lf^*$, $r_i^*$, $Lr_{i,\#}$, $Lc^*$ and
$Lc_{\#}$ for the functors between the categories $DM^{eff}$ which are
defined by $Lf^*$, $Lc^*$ and $Lc_{\#}$ respectively. Note
(cf. \cite[Prop. 2.6.2]{Delnotes}) that these functors have the same
adjunction properties as the original functors.
\begin{lemma}
\llabel{isconserv2}
The family of functors
$$r_i^*:DM^{eff}_{-}({\cal X})\sr DM^{eff}_{-}({\cal X}_i)$$
is conservative i.e. if $r_i^*(K)\cong 0$ for all $i$ then $K\cong 0$.
\end{lemma}
\begin{proof}
Same as in Lemma \ref{isconserv}.
\end{proof}
\begin{proposition}
\llabel{a1oo}
The tensor product $\stackrel{L}{\oo}$ respects $\af$-equivalences.
\end{proposition}
\begin{proof}
It is enough to show that for $K\in r_{i,\#}(W^{el}({\cal X}_i))$ and
any $L$ the object  $K\stackrel{L}{\oo} L$ is zero in
$DM^{eff}_{-}({\cal X})$. By Lemma \ref{isconserv2} it is sufficient
to show that $r_{j}^*(K)\cong 0$ for all $j$. This follows immediately
from Lemma \ref{rioo} and (\ref{rirj}).
\end{proof} 
By Proposition \ref{a1oo}, the tensor structure on $D({\cal X})$ defines a
tensor structure on $DM^{eff}_{-}({\cal X})$. Since any distinguished
triangle in $DM^{eff}_{-}({\cal X})$ is, by definition, isomorphic to
the image of a distinguished triangle in $D({\cal X})$, Proposition
\ref{mayoo} implies immediately the following result.
\begin{proposition}
\llabel{mayoo} The axioms (TC1)-(TC3) of \cite{MayTT} hold for
$DM^{eff}_{-}({\cal X})$.
\end{proposition}
\begin{proposition}
\llabel{summands}
The category $DM^{eff}_{-}({\cal X})$ is Karoubian i.e. projectors in
this category have kernels and images.
\end{proposition}
\begin{proof}
For any $K$ in $DM^{eff}_{-}({\cal X})$ the countable direct sum
$\oplus_{i=1}^{\infty} K$ exists in $DM^{eff}_{-}({\cal X})$ for
obvious reasons. This implies the statemnt of the proposition in view
of the following easy generalization of \cite[Prop.1.6.8 p.65]{Nee}.
\end{proof}
\begin{lemma}
\llabel{neegen}
Let $D$ be a triangulated category such that for any object $K$ in $D$
the countable direct sum $\oplus_{i=1}^{\infty} K$ exists. Then $D$ is
Karoubian.
\end{lemma}
\begin{proof}
Same as the proof of \cite[Prop.1.6.8 p.65]{Nee}.
\end{proof}

\subsection{Relative Tate motives}
\llabel{ss4}
For any $S$ we may define the Tate objects $\zz(p)[q]$ in
$DM_{-}^{eff}(S)$ in the same way they were defined in
\cite[p.192]{collection} for $S=Spec(k)$. For $\cal X$ over $S$ we
define the Tate objects $\zz(p)[q]$ in $DM_{-}^{eff}({\cal X})$ as
$Lc^*(\zz(p)[q])$. Note that this definition does not depend on $S$ -
one may always consider $\cal X$ as a simplicial scheme over
$Spec(\zz)$ and lift the Tate objects from $Spec(\zz)$. We denote by
$DT({\cal X})$ the thick subcategory in $DM^{eff}_{-}({\cal X})$
generated by Tate objects i.e. the smallest subcategory which is
closed under shifts, triangles and direct summands and contains
$\zz(i)$ for all $i\ge 0$. One should properly call it the
triangulated category of effective Tate motives of finite type over
$\cal X$ but we will call it simply the category of Tate motives over
$\cal X$. When $\cal X$ is clear from the context we will write $DT$
instead of $DT({\cal X})$.

The subcategory $DT({\cal X})$ is clearly closed under the tensor
product and Proposition \ref{mayoo}
implies that $DT({\cal X})$ is a tensor triangulated category
satisfying May's axiom $TC3$.
\begin{remark}\rm
The category $DT({\cal X})$ does not coincide in general with the
triangulated subcategory generated in $DM^{eff}_{-}({\cal X})$ by Tate
objects. Consider for example the case when ${\cal X}={\cal X}_1\amalg
{\cal X}_2$ and both ${\cal X}_1$ and ${\cal X}_2$ are non-empty. Then
the constant presheaf with transfers $\zz$ is a direct sum of
$\zz_{1}$ and $\zz_2$ where $\zz_i$ is the constant presheaf with
transfers on ${\cal X}_i$. One can easily show that $\zz_i$'s are not
in the triangulated subcategory generated by Tate objects. 

However, one can show that the problem demonstrated by this example is
the only possible one - if $H^{0}({\cal X},\zz)$ is $\zz$ then the
triangulated subcategory in $DM^{eff}_{-}({\cal X})$ generated by Tate
objects  is closed under directs summands and therefore coincides with
$DT({\cal X})$.
\end{remark}
For $M$ in $DT$ we denote as usually by $H_{*,*}(M)$ the groups
$$H_{p,q}(M)=\left\{
\begin{array}{ll}
Hom(\zz,M(-q)[-p])&\mbox{\rm for $q\le 0$}\\
Hom(\zz(q)[p], M)&\mbox{\rm for $q\ge 0$}
\end{array}
\right.
$$
and by $H^{*,*}(M)$ the groups
$$H^{p,q}(M)=\left\{
\begin{array}{ll}
Hom(M,\zz(q)[p])&\mbox{\rm for $q\ge 0$}\\
0&\mbox{\rm for $q< 0$}
\end{array}
\right.
$$
\begin{lemma}
\llabel{start}
Let $f:M\sr M'$ be a morphism in $DT$ which defines isomorphisms on the
groups $H_{p,q}(-)$ for $q\ge 0$. Then $f$ is an isomorphism.
\end{lemma}
\begin{proof}
For a given $f$, the class of $N$ such that the maps
$$Hom(N[p],M)\sr Hom(N[p],M')$$
are isomorphisms for all $p$ is a thick subcategory of $DT$.  
Our condition means that it contains all $\zz(q)$. Therefore it
coincides with the whole $DT$ and we conclude that $f$ is an
isomorphism by Yoneda Lemma.
\end{proof}
Let $\cal X$ be a smooth simplicial scheme over $S$. For such an $\cal
$ we define $M({\cal X})$ as the object in $DM_{-}^{eff}(S)$ given by
the complex $\zz_{tr}({\cal X})$ associated with the simplicial object
${\cal X}$ in $SmCor(S)$. Note that this definition is compatible with
the definition of motives of smooth simplicial schemes given in
\cite{MCnew}.
\begin{proposition}
\llabel{main1} For ${\cal X}$ as above, there are natural
isomorphisms:
$$Hom_{DM({\cal X})}(\zz(q')[p'],
\zz(q)[p])=Hom_{DM(S)}(M({\cal X})(q')[p'], \zz(q)[p]))$$
\end{proposition}
\begin{proof}
We have by adjunction
$$Hom_{DM({\cal X})}(\zz(p')[q'],
\zz(q)[p])=Hom_{DM({\cal X})}(c^*\zz(q')[p'], c^*\zz(q)[p]))=$$
$$=Hom_{DM(S)}(Lc_{\#}c^*\zz(q')[p'], \zz(q)[p])$$
and Proposition \ref{projform} implies that for any $M$ in
$DM^{eff}_{-}(S)$ one has
$$Lc_{\#}c^*(M)=M({\cal X})\oo M.$$
\end{proof}
\begin{cor}
\llabel{noneg}
For $\cal X$ as above and any $i>0$ one has 
$$Hom_{DM({\cal X})}(\zz,\zz[-i])=0$$
\end{cor}
Combining Proposition \ref{main1} with the Cancellation Theorem
\cite{cancellation} we get the following result.
\begin{cor}
\llabel{cancell}
Let now $\cal X$ be a smooth simplicial scheme over a perfect field
$k$. Then one has
$$
Hom_{DM({\cal X})}(\zz(q')[p'],
\zz(q)[p])=$$
$$=\left\{
\begin{array}{cc}
0 &\mbox{\rm for $q<q'$}\\
Hom_{DM({\cal X})}(\zz,
\zz(q-q')[p-p'])&\mbox{\rm for $q\ge q'$}
\end{array}
\right.
$$
\end{cor}
\begin{remark}\rm
Using the fact that a regular scheme of equal characteristic is the
inverse limit of a system of smooth schemes over a perfect field it is
easy to generalize Corollary \ref{cancell} to smooth simplicial
schemes over regular schemes of equal characteristic. We expect it
hold for all regular simplicial schemes but not for general
(simplicial) schemes.
\end{remark} 
Starting from this point we assume that $S=Spec(k)$ where $k$ is a
perfect field and $\cal X$ is a smooth simplicial scheme over $S$.
\begin{lemma}
\llabel{inthomt}
For any $X$, $Y$ in $DT({\cal X})$ there exists an internal Hom-object
$(Z,e)$ from $X$ to $Y$.
\end{lemma}
\begin{proof}
Consider first the case when $X=\zz(i)$ and $Y=\zz(j)$. Corollary
\ref{cancell} implies immediately that $(0,0)$ is an internal
Hom-object from $\zz(i)$ to $\zz(j)$ for $j<i$. The same corollary
shows that $(\zz(j-i),e)$, where $e$ is the isomorphism
$\zz(j-i)\oo\zz(i)\sr\zz(j)$, is an internal Hom-object from $\zz(i)$
to $\zz(j)$ for $j\ge i$. The fact that $(Z,e)$ exists for arbitrary
$X$ and $Y$ follows now from Theorem \ref{appmain} and the obvious
argument for direct summands. 
\end{proof}
Starting from this point we choose a specification of internal
Hom-objects in $DT({\cal X})$ (see Appendix) such that for
$i\ge j$ one has $\uu{Hom}(\zz(j),\zz(i))=\zz(i-j)$.

Let $DT_{\ge n}$ (resp. $DT_{< n}$) be the thick subcategory in
$DT({\cal X})$ generated by Tate objects $\zz(i)$ for $i\ge n$
(resp. $i< n$). The subcategories $DT_{\ge n}$ form a decreasing filtration
$$\dots \subset DT_{\ge 1}\subset  DT_{\ge 0}=DT({\cal X})$$
and we have
$$\cap_n DT_{\ge n}=0$$
Similarly the subcategories $DT_{< n}$ form an increasing filtration
$$0=DT_{< 0}\subset DT_{< 1}\subset\dots\subset DT({\cal X})$$%
and we have
$$\cup_n DT_{< n}=DT({\cal X})$$
We call these filtrations {\em the slice filtrations} on $DT$ since
they are similar to the slice filtration on the motivic stable
homotopy category. Since we consider here only Tate motives the slice
filtration coincides (up to numbering) with the weight filtration but
for more general motives they are different. 
\begin{lemma}
\llabel{or1}
Let $M$ be such that $H_{*,i}(M)=0$ for all $i\ge n$. Then $M$ lies in
$DT_{< n}$.
\end{lemma}
\begin{proof}
Set
$$\Psi(M)=\uu{Hom}(\uu{Hom}(M,\zz(n-1)),\zz(n-1))$$
The adjoint to the morphism
$$ev\circ \sigma:M\oo \uu{Hom}(M,\zz(n-1))\sr \zz(n-1)$$
where $\sigma$ is the permutation of multiples is a morphism
$$\psi:M\sr \Psi(M)$$
which is natural in $M$. Using Proposition \ref{homex} and Corollary
\ref{cancell} one verifies immediately that $\Psi(M)$ lies in $DT_{<
n}$ for $M=\zz(q)[p]$, $i\ge 0$ and therefore, by Proposition
\ref{homex}, $\Psi(M)$ lies in $DT_{< n}$ for all $M$.  It remains to
check that for $M$ satisfying the condition of the lemma $\psi$ is an
isomorphism. Consider the maps
$$H_{*,i}(M)\sr H_{*,i}(\Psi(M)).$$
For $i<n$ and $M=\zz(q)[p]$ these maps are isomorphisms by Corollary
\ref{cancell}. Together with Proposition \ref{homex} and the five
lemma we conclude that they are isomorphisms for $i<n$ and all $M$.
On the other hand $H_{*,i}(\Psi(M))=0$ for $i\ge n$ and any $M$ and
therefore under the conditions of the lemma $\psi$ is an isomorphism
by Lemma \ref{start}.
\end{proof}
\begin{lemma}
\llabel{sf1}
For any $M$ in $DT$ and any $n$ there exists a distinguished triangle 
of the form
\begin{eq}\llabel{stm}
\Pi_{\ge n}M\sr M\sr \Pi_{<n}M\sr\Pi_{\ge n}M[1]
\end{eq}
such that $\Pi_{\ge n}M$ lies in $DT_{\ge n}$ and $\Pi_{<n}M$ lies in
$DT_{< n}$.
\end{lemma}
\begin{proof}
Set
$$\Pi_{\ge n}(M)=\uu{Hom}(\zz(n),M)(n)$$
and define $\Pi_{<n}M$ by the distinguished triangle
$$\Pi_{\ge n}M\sr M\sr \Pi_{<n}M\sr\Pi_{\ge n}M[1]$$
where the first arrow is $e=ev_{\zz(n),M}$. Clearly, $\Pi_{\ge n}M$ lies
in $DT_{\ge n}$. It remains to check that $\Pi_{<n}M$ lies in $DT_{<
n}$. By Lemma \ref{or1} it is sufficient to check that
$H_{*,i}(\Pi_{<n}M)=0$ for all $i\ge n$ i.e. that $e$ defines an
isomorphism on $H_{*,i}(-)$ for $i\ge n$. In view of Proposition
\ref{homex} and the Five Lemma it is sufficient to verify it for
$M=\zz(q)[p]$ in which case it follows from Corollary \ref{cancell}.
\end{proof}
\begin{remark}\rm
Note that the proof of Lemma \ref{or1} shows that in the distinguished
triangle of Lemma \ref{sf1}
one may choose $M\sr \Pi_{<n}M$ to be 
$$\psi:M\sr \uu{Hom}(\uu{Hom}(M,\zz(n-1)),\zz(n-1))$$
\end{remark}
\begin{lemma}
\llabel{sf2}
Let $f:M_1\sr M_2$ be a morphism in $DT$ and let
$$\Pi_{\ge n}M_1\sr M_1\sr \Pi_{<n}M_1\sr\Pi_{\ge n}M_1[1]$$
$$\Pi_{\ge n}M_2\sr M_2\sr \Pi_{<n}M_2\sr\Pi_{\ge n}M_2[1]$$
be distinguished triangles satisfying the conditions of Lemma
\ref{sf1}. Then there exists a unique morphism of triangles of the
form
\begin{eq}\llabel{sf2eq}
\begin{CD}
\Pi_{\ge n}M_1 @>>> M_1@>>> \Pi_{<n}M_1@>>>\Pi_{\ge n}M_1[1]\\
@VVV @VfVV @VhVV @VVV\\
$$\Pi_{\ge n}M_2@>>> M_2@>>> \Pi_{<n}M_2@>>>\Pi_{\ge n}M_2[1]
\end{CD}
\end{eq}
\end{lemma}
\begin{proof}
The uniqueness follows from the fact that 
$$Hom(\Pi_{\ge n}M_1[*],\Pi_{<n}M_2)=0$$
The same fact implies that
$$Hom(M_1,\Pi_{<n}M_2)=Hom(\Pi_{<n}M_1,\Pi_{<n}M_2)$$
and therefore there exists a morphism $h$ which makes the middle square of
(\ref{sf2eq}) commutative. Extending this square to a morphism of
distinguished triangles we get the existence part of the lemma.
\end{proof}
\begin{lemma}
\llabel{sf3}
For any $M$ and any triangle of the form (\ref{stm}) satisfying the
conditions of Lemma \ref{sf1} one has:
\begin{enumerate}
\item For any $N$ in $DT_{<n}$ one has
$$Hom(\Pi_{<n}M, N)=Hom(M,N)$$
$$Hom(\Pi_{\ge n}M,N)=0$$
\item For any $N$ in $DT_{\ge n}$ one has
$$Hom(N,\Pi_{\ge n}M)=Hom(N,M)$$
$$Hom(N,\Pi_{<n}M)=0$$
\end{enumerate}
\end{lemma}
\begin{remark}\rm
The major difference between the slice filtrations in the tirangulated
category of motives and in the motivic stable homotopy category is
that in the later case Lemma \ref{sf3} does not hold. For $N$ in
$SH_{\ge n}$ and $M$ in $SH_{<n}$ one may have $Hom(N,M)\ne 0$. The
Hopf map $S^1_t\sr S^0$ is an example of a morphism of such a form.
\end{remark}
Lemma \ref{sf2} implies that the triangles of the form (\ref{stm}) are
functorial in $M$. Choosing one such triangle for each $M$ and each
$n$ we get functors:
$$\Pi_{\ge n}:DT\sr DT_{\ge n}$$
$$\Pi_{<n}:DT\sr DT_{<n}$$
Lemma \ref{sf3} shows that $\Pi_{\ge n}$ is a right adjoint to the
corresponding inclusion and $\Pi_{<n}$ is a left adjoint to the
corresponding inclusion. We can also describe these functors in terms
of internal Hom-functors
$$\Pi_{\ge n}(M)=\uu{Hom}(\zz(n),M)(n)$$
$$\Pi_{< n}(M)=\uu{Hom}(\uu{Hom}(M,\zz(n-1)),\zz(n-1))$$
By Proposition \ref{homex} we conclude that $\Pi_{\ge n}$ and $\Pi_{<
n}$ are triangulated functors.

Applying Lemma \ref{sf3} for $N=\Pi_{\ge (n+1)}M$ and $N=\Pi_{<(n-1)}$
we get canonical morphisms
$$\Pi_{\ge (n+1)}M\sr \Pi_{\ge n}M$$
$$\Pi_{<n}\sr \Pi_{<(n-1)}M$$
We extend these morphisms to distinguished triangles
\begin{eq}\llabel{sfup}
\Pi_{\ge (n+1)}M\sr \Pi_{\ge n}M\sr s_n(M)\sr \Pi_{\ge (n+1)}M[1]
\end{eq}
\begin{eq}\llabel{sfdown}
s'_{n-1}(M)\sr \Pi_{<n}\sr \Pi_{<(n-1)}M\sr s'_{n-1}(M)[1]
\end{eq}
One observes easily that $s'_n(M)\cong s_n(M)$ and that this object
lies in $DT_n=DT_{\ge n}\cap DT_{< n+1}$. Therefore, Lemma \ref{sf2}
is applicable to triangles (\ref{sfup}), (\ref{sfdown}) and we
conclude that these triangles are functorial in $M$. Choosing one such
triangle for each $M$ and each $n$ we obtain functors 
\begin{eq}
\llabel{sln}
s_n:DT\sr DT_n
\end{eq}
Since $s_n=\Pi_{< n+1}\Pi_{\ge n}$ these functors are triangulated.
We set
\begin{eq}
\llabel{sl}
s_*=\oplus_{n\ge 0} s_n:DT\sr \oplus_{n\ge 0} DT_n
\end{eq}
Note that (\ref{sl}) makes sense since for any $M$ one has $s_n(M)=0$
for all but finitely many $n$. The functors (\ref{sln}), (\ref{sl})
are called the slice functors over $\cal X$.  
\begin{lemma}
\llabel{slconserv}
The functor $s_*$ is concervative i.e. if $s_*(M)=0$ then $M=0$.
\end{lemma}
\begin{proof}
Follows easily by induction.
\end{proof}
\begin{lemma}
\llabel{sltensor} Define tensor product on $\oplus_n DT_n$ by the
formula
$$(M_i)_{i\ge 0}\oo (M_j)_{j\ge 0}=(\oplus_{i+j=n} M_i\oo M_j)_{n\ge
0}.$$
Then for any $N, M$ there is a natural isomorphism 
$$s_*(N\oo M)=s_*(N)\oo
s_*(M).$$
\end{lemma}
\begin{proof}
For any $M$ and $N$ the morphisms $\Pi_{\ge i}M\sr M$ and $\Pi_{\ge
j}N\sr N$ define a morphism
\begin{eq}
\llabel{mor1}
s_{i+j}(\Pi_{\ge i}M\oo \Pi_{\ge
j}N)\sr s_{i+j}(M\oo N)
\end{eq}
On the other hand the the projections $\Pi_{\ge i}M\sr s_i(M)$ and $\Pi_{\ge
j}N\sr s_j(N)$ define a morphism
\begin{eq}
\llabel{mor2}
s_{i+j}(\Pi_{\ge i}M\oo \Pi_{\ge
j}N)\sr s_{i+j}(s_i(M)\oo s_j(N))=s_i(M)\oo s_j(N)
\end{eq}
One can easily see that (\ref{mor2}) is an isomorphism. The inverse to
(\ref{mor2}) together with (\ref{mor1}) defines a natural morphism
$$\oplus_{i+j=n} s_i(M)\oo s_j(N)\sr s_{n}(M\oo N)$$
One verifies easily that it is an isomorphism for $M=\zz(q)[p]$,
$N=\zz(q')[p']$ which implies by the Five Lemma that it is an
isomorphism for all $M$ and $N$.
\end{proof}
\begin{lemma}
\llabel{slinv}
The functors $\Pi_{\ge n}$, $\Pi_{<n}$ and $s_*$ commute with the
pull-back functors $Lf^*$ for arbitrary morphisms of smooth simplicial
schemes $f:{\cal X}\sr {\cal Y}$.
\end{lemma}
\begin{proof}
This follows immediately from the fact that the functor $Lf^*$ takes
$DT_{\ge n}$ to $DT_{\ge n}$ and $DT_{<n}$ to $DT_{<n}$.
\end{proof}
\begin{lemma}
\llabel{dualten} Let $X$, $Y$, $P_1$, $P_2$ be such that for some $n$
and $m$ one has
$$X\in DT_{\le n}\,\,\,\,\, P_1\in DT_{\ge n}$$
$$Y\in DT_{\le m}\,\,\,\,\, P_2\in DT_{\ge m}$$
Then 
$$(\uu{Hom}(X,P_1)\oo\uu{Hom}(Y,P_2), ev_{X,P_1}\oo ev_{Y,P_2})$$
is an internal Hom-object from $X\oo Y$ to $P_1\oo P_2$.
\end{lemma}
\begin{proof}
We need to verify that for any $M$ the homomorphism
$$Hom(M, X'\oo Y')\sr Hom(M\oo X\oo Y, P_1\oo P_2)$$
defined by $ev_{X,P_1}\oo ev_{Y,P_2}$ is a bijection. Using
Proposition \ref{homex} and the Five Lemma we can reduce the problem
to the case when $M, X, Y, P_1$ and $P_2$ are all motives of the form
$\zz(q)[p]$ with the appropriate restrictions of $q$. In this case the
statement follows from Corollary \ref{cancell}.
\end{proof}
\begin{lemma}
\llabel{unext}
Let $n\ge 0$ be an integer and 
\begin{eq}
\llabel{ab}
M_0\stackrel{a}{\sr} M_1\stackrel{b}{\sr} M_2
\end{eq}
a sequence of morphisms in $DT$ such that the following conditions
hold:
\begin{enumerate}
\item $M_0$ is in $DT_{\ge n}$ and $s_{i}(a)$ is an isomorphism for
$i\ge n$
\item $M_2$ is in $DT_{<n}$ and $s_i(b)$ is an isomorphism for $i<n$.
\end{enumerate}
Then there exists a unique morphism $M_2\sr M_0[1]$ such that the
sequence
$$M_0\stackrel{a}{\sr} M_1\stackrel{b}{\sr} M_2\sr M_0[1]$$
is a distinguished triangle. This distinguished triangle is then
isomorphic to the triangle
$$\Pi_{\ge n}(M_1)\sr M_1\sr \Pi_{<n}(M_1)\sr \Pi_{\ge n}(M_1)[1]$$
\end{lemma}
\begin{proof}
Note first that $Hom(M_0,M_2)=0$ and therefore $b\circ a=0$. Extending
$a$ to a distinguished triangle we get a factorization of $b$ through
a morphism $\phi:cone(a)\sr M_2$. Our conditions imply that
$s_*(\phi)$ is an isomorphism and we conclude by Lemma \ref{slconserv}
that $\phi$ is an isomorphism and hence (\ref{ab}) extends to a
distinguished triangle. The proof of two other statements of the lemma
is straightforward. 
\end{proof}
Since the functor $X\mapsto X(n)$ from $DT_0$ to $DT_n$ is an
equivalence (by Corollary \ref{cancell}) we may consider $s_*$ as a
functor with values in $\oplus_n DT_0$. To describe the category
$DT_0$ consider the projection
\begin{eq}
\llabel{projto}
D({\cal X})\sr DM_{-}^{eff}({\cal X})
\end{eq}
from the derived category of presheaves with transfers over $\cal X$
to $DM$. Let us say that a presheaf with transfers $(F_i)$ on $\cal X$
is locally constant if for every $i$ the presheaf with transfers $F_i$
on $Sm/{\cal X}_i$ is locally constant. Locally constant presheaves
with transfers clearly form an abelian subcategory $LC$ in the abelian
category of presheaves with transfers. 
\begin{remark}\rm
Let $X\mapsto CC(X)$ be the functor which commutes with coproducts and
takes a connected scheme to the point. Applying $CC$ to a simplicial
scheme $\cal X$ we get a simplicial set $CC({\cal X})$. If $CC({\cal
X})$ is a connected simplicial set then $LC({\cal X})$ is
equivalent to the category of modules over $\pi_1(CC({\cal X}))$.
\end{remark}
Let $DLC$ be full the subcategory in $D({\cal X})$ which consists of
complexes of presheaves with transfers with locally constant
cohomology presheaves. Note that $DLC$ is a thick subcategory. Let
further $DT'_0$ be the thick subcategory in $DLC$ generated by the
constant sheaf $\zz$. Note that the category $DLC$ is Karoubian and
therefore the same holds for $DT'_0$.
\begin{proposition}
\llabel{dt00}
The projection (\ref{projto}) defines an equivalence between $DT_0'$
and $DT_0$.
\end{proposition}
\begin{proof}
Let us show first that the restriction of (\ref{projto}) to $DT_0'$ is
a full embedding. In order to do this we have to show that objects of
$DT'_0$ are orthogonal to objects of $r_i(W^{el}({\cal X}_i))$. In
order to do this it is enough to show that for a smooth scheme $X$ the
constant presheaf with transfers is orthogonal to complexes lying in
$W^{el}(X)$ i.e. that for any such presheaf $F$ and such a complex $K$
one has
$$Hom_{D}(K,F)=0$$
This follows immediately from the fact that for a smooth $X$ and
constant $F$ one has
$$H^i_{Nis}(X,F)=0\,\,\,\rm for \,\,\, i>0$$
and
$$F(X\times\af)=F(X).$$
To finish the proof of the proposition it remains to check that the
image of $DT'_0$ lies in $DT$ and that any object of $DT_0$ is
isomorphic to the image of an object in $DT_0'$. The first statement
is obvious from definitions.  To see the second one observe that since
our functor is a full embedding and the source is Karoubian its image
is a thick subcategory. Since it contains $\zz$ it coincides
with $DT_0$.
\end{proof} 
\begin{remark}\rm
The category $DLC$ has a t-structure whose heart is the category
$LC({\cal X})$ of locally constant presheaves with transfers. It is
equivalent to the derived category of $LC({\cal X})$ if and only if
$CC({\cal X})$ is a $K(\pi,1)$.
\end{remark}
\begin{remark}\rm
The condition that the terms of $\cal X$ are disjoint unions of smooth
schemes over a field is important for Proposition \ref{dt00}. More
precisely what is required is that the terms of $\cal X$ are disjoint
unions of geometrically unibranch (e.g.normal) schemes. If this
condition does not hold the Nisnevich cohomology of the terms with the
coefficients in constant sheaves may be non-zero.
\end{remark}

\subsection{Embedded simplicial schemes}
\llabel{secemb}
\llabel{ss5}
In this section we consider a special case of the general theory
developed above.
\begin{definition}
\llabel{embed} A smooth simplicial scheme $\cal X$ over $k$ is called
embedded (over $k$) if the morphisms
$$M({\cal X}\times {\cal X})\sr M({\cal X})$$
defined by the projections are isomorphisms.
\end{definition}
\begin{lemma}
\llabel{opyat} Let $\cal X$ be a smooth simplicial scheme over $S$
such that ${\cal X}\sr \pi_0({\cal X})$ is a local equivalence and the
morphism $\pi_0({\cal X})\sr pt$ is a monomorphism. Then $\cal X$ is
embedded. 
\end{lemma}
\begin{proof}
Our conditions imply that $pr:{\cal X}\times{\cal X}\sr {\cal X}$
is a local equivalence. Therefore, $M(pr)$ is an isomorphism.
\end{proof}
\begin{example}
\llabel{embedex}\rm
Let $X$ be a smooth scheme over $S$ and $\check{C}(X)$ the Chech
simplicial scheme of $X$ (see \cite[Sec.9]{MCnew}). Then $\check{C}(X)$ is
embedded. The sheaf $\pi_0(X)$ takes a smooth connected scheme $U$ to
$pt$ if for any point $p$ of $U$ there exists a morphism 
$$Spec({\cal O}^h_{U,p})\sr X$$
and to $\emptyset$ otherwise.
\end{example}
\begin{example}
\llabel{resol}\rm For any subpresheaf $F$ of the constant sheaf $pt$
the standard simplicial resolution $G(F)$ of $F$ is an embedded
simplicial scheme. We will show below that for any embedded $\cal X$
there exists $F\subset pt$ such that $M({\cal X})\cong M(G(F))$. See
Lemma \ref{descr}.
\end{example} 
\begin{lemma}
\llabel{rlminusone}
Let $\cal X$ be an embedded simplicial scheme and 
$$a:c^*Lc_{\#}c^*\sr c^*$$
the natural transformation defined by the adjunction. Then $a$ is an
isomorphism. 
\end{lemma}
\begin{proof}
By the definition of adjoint functors the obvious map
$b:c^*\sr c^*Lc_{\#}c^*$ is a section of $a$. Hence, it is sufficient
to show that
$$b\circ a:c^*Lc_{\#}c^*\sr c^*Lc_{\#}c^*$$
is the identity. This map is adjoint to the map
$$p_1:Lc_{\#}c^*Lc_{\#}c^*\sr Lc_{\#}c^*$$
which collapses the second copy of the composition $Lc_{\#}c^*$ to the
identity. On the other hand the identity on $c^*Lc_{\#}c^*$ is adjoint
in the same way to the map
$$p_2:Lc_{\#}c^*Lc_{\#}c^*\sr Lc_{\#}c^*$$
which collapses the first copy of the composition $Lc_{\#}c^*$ to the
identity. It remains to show that $p_1=p_2$. By Proposition
\ref{projform} we have
$$Lc_{\#}c^*(N)=N\oo M({\cal X})$$
and one can easily see that $p_1$ and $p_2$ can be identified with the
morphisms 
$$N\oo M({\cal X})\oo M({\cal X})\sr N\oo M({\cal X})$$
defined by the two projections
$$M({\cal X})\oo M({\cal X})=M({\cal X}\times {\cal X})\sr M({\cal
X})$$
These two projections are isomorphisms by our assumption on $\cal X$
and since the diagonal is their common section we conclude that they
are equal.
\end{proof}
\begin{lemma}
\llabel{rl0} Let $\cal X$ be an embedded simplicial scheme. Then for
any object $M$ of $DM^{eff}_{-}(S)$ and any object $N$ of
$DM^{eff}_{-}({\cal X})$ the natural map
\begin{eq}
\llabel{rl0eq}
Hom(c^*M,N)\sr Hom(Lc_{\#}c^*M,Lc_{\#}N)
\end{eq}
is bijective.
\end{lemma}
\begin{proof}
By adjunction, the right hand side of (\ref{rl0eq}) can be identified
with $Hom(c^*Lc_{\#}c^*M,N)$ and with respect to this identification
the map (\ref{rl0eq}) is defined the natural transformation
$a:c^*Lc_{\#}c^*\sr c^*$. The statement of the lemma follows now from
Lemma \ref{rlminusone}.
\end{proof}
For any $\cal X$ let $DM_{\cal X}$ denote the localizing subcategory
in $DM^{eff}_{-}({\cal X})$ which is generated by objects of the form
$c^*(M)$ for $M$ in $DM_{-}^{eff}(S)$. Note that $DM_{\cal X}$
contains the category $DT({\cal X})$ of Tate motives over $\cal
X$. Lemma \ref{rl0} immediately implies the following result.
\begin{lemma}
\llabel{rl1}
If $\cal X$ is embedded then $Lc_{\#}$ defines a full embedding
$$Lc_{\#}:DM_{\cal X}\sr DM_{-}^{eff}(S)$$
\end{lemma}
\begin{lemma}
\llabel{rl2}
If $\cal X$ is embedded and $M,N$ are objects of $DM_{\cal X}$ then
the canonical morphism
\begin{eq}
\llabel{embten}
Lc_{\#}(M\oo N)\sr Lc_{\#}(M)\oo Lc_{\#}(N)
\end{eq}
is an isomorphism.
\end{lemma}
\begin{proof}
Note fisrt that a natural morphism of the form (\ref{embten}) is
defined by adjunction since $c^*$ commutes with the tensor
products. Since both sides of (\ref{embten}) are triangulated functors
in each of the arguments the class of $M$ and $N$ such that
(\ref{embten}) is an isomorphism is a localizing subcategory. It
remains to check that it contains pairs of the form $c^*M(X)$,
$c^*M(Y)$ where $X$, $Y$ are smooth schemes over $S$. With respect to
isomorphisms
$$Lc_{\#}c^*(M(X))=M(X)\oo M({\cal X})$$
$$Lc_{\#}c^*(M(Y))=M(Y)\oo M({\cal X})$$
$$Lc_{\#}c^*(M(X)\oo M(Y))=M(X)\oo M(Y)\oo M({\cal X})$$
the morphism (\ref{embten}) coincides with the morphism
$$M(X)\oo M(Y)\oo M({\cal X})\sr M(X)\oo M({\cal X})\oo M(Y)\oo
M({\cal X})$$
defined by the diagonal of $\cal X$. This morphism is an isomorphism
since $\cal X$ is embedded.
\end{proof}
Lemma \ref{rl2} shows that the restriction of $Lc_{\#}$ to $DM_{\cal
X}$ is almost a tensor functor. Note that it is not really a tensor
functor since 
$$Lc_{\#}(\zz)=M({\cal X})\ne \zz.$$
We also have to distinguish the internal Hom-objects in $DM_{\cal X}$
and $DM^{eff}_{-}(S)$. See Example \ref{inthomex} below.

Starting from this point we assume that ${\cal X}$ is embedded over
$S$. We use Lemma \ref{rl1} to identify $DM_{\cal X}$ with a full
subcategory in $DM_{-}^{eff}(S)$. With respect to this identification
the functor $c^*$ takes $M$ to $M\oo M({\cal X})$.
\begin{lemma}
\llabel{donotneed}
The subcategory $DM_{\cal X}$ coincides with the subcategory of
objects $M$ such that the morphism 
\begin{eq}
\llabel{mproj}
M\oo M({\cal X})\sr M
\end{eq}
is an isomorphism.
\end{lemma}
\begin{proof}
Let $D$ be the subcategory of $M$ such that (\ref{mproj}) is an
isomorphism. As was mentioned above the functor $c^*$ takes a motive
$M$ to $M\oo M({\cal X})$ so $D$ is contained in $DM_{\cal X}$. Since
$D$ is a localizing subcategory and $DM_{\cal X}$ is generated by
motives of the form $M(X)\oo M({\cal X})$ the opposite inclusion
follows from the fact that for any $X$ the morphism
$$M(X)\oo M({\cal X})\oo M({\cal X})\sr M(X)\oo M({\cal X})$$
is an isomorphism.
\end{proof}
\begin{remark}\rm
Lemma \ref{donotneed} show that $DM_{\cal X}$ is an ideal in
$DM^{eff}_{-}(S)$ i.e. for any $K$ and any $M$ in $DM_{\cal X}$ the
tensor product $K\oo M$ is in $DM_{\cal X}$.
\end{remark}
\begin{lemma}
\llabel{rl3}
For $M$ in $DM_{\cal X}$ and $N\in DM^{eff}_{-}(S)$ the natural map
\begin{eq}
\llabel{ebi}
Hom(M, N\oo M({\cal X}))\sr Hom(M,N)
\end{eq}
is an isomorphism.
\end{lemma}
\begin{proof}
Consider the map
$$Hom(M,N)\sr Hom(M, N\oo M({\cal X}))$$
which takes $f$ to $(f\oo Id_{M({\cal X})})\circ \phi^{-1}$ where
$\phi$ is the morphism of the form (\ref{mproj}). One can easily see
that this map is both the right and the left inverse to (\ref{ebi}).
\end{proof}
Lemma \ref{rl3} has the following straightforward corollary.
\begin{lemma}
\llabel{intres0}
Let $M,N$ be objects of $DM_{\cal X}$, $P$ an object of
$DM^{eff}_{-}(S)$ and 
$$e:M\oo N\sr P$$
a morphism such that $(N,e)$ is an internal Hom-object from
$M$ to $P$ in $DM^{eff}_{-}(S)$. Define $e_{\cal X}$ as the 
morphism
$$e_{\cal X}:M\oo N\sr P\oo M({\cal X})$$
corresponding to $e$ by Lemma \ref{rl3}. Then $(N,e_{\cal X})$ is an
internal Hom-object from $M$ to $P\oo M({\cal X})$ in $DM_{\cal X}$.
\end{lemma}
\begin{definition}
\llabel{restr}
An object $M$ in $DM_{\cal X}$ is called restricted if for any $N$
in $DM^{eff}_{-}(S)$ the natural map
\begin{eq}
\llabel{restreq}
Hom(N, M)\sr Hom(N\oo M({\cal X}), M)
\end{eq}
is an isomorphism.
\end{definition}
For our next result we need to recall the motivic duality theorem.
For a smooth variety $X$ over a field $k$ and a smooth subvariety $Z$
of $X$ of pure codimension $d$ the Gysin distinguished triangle
defines the motivic cohomology class of $Z$ in $X$ of the form
$M(X)\sr \zz(d)[2d]$. In particular for $X$ of pure dimension $d$ the
diagonal gives a morphism $M(X)\oo M(X)\sr \zz(d)[2d]$ which we denote
by $\Delta^*$. The following motivic duality theorm is proved in
\cite[]{}.
\begin{theorem}
\llabel{motdual} For a smooth projective variety $X$ of pure dimension
$d$ over a perfect field $k$ the pair $(M(X),\Delta^*)$ is the
internal Hom-object from $M(X)$ to $\zz(d)[2d]$ (see \ref{appendix1}).
\end{theorem}
\begin{lemma}
\llabel{spcase} Let $S=Spec(k)$ where $k$ is a perfect field and $X$
be a smooth projective variety such that $M(X)$ lies in $DM_{\cal
X}$. Then $M(X)$ is restricted.
\end{lemma}
\begin{proof}
We may clearly assume that $X$ has pure dimension $d$ for some $d\ge
0$. Theorem \ref{motdual} implies that for $M=M(X)$ the morphism
(\ref{restreq}) is isomorphic to the morphism
$$Hom(N\oo M(X),\zz(d)[2d])\sr Hom(N\oo M({\cal X})\oo
M(X),\zz(d)[2d])$$
which is an isomorphism by Lemma \ref{donotneed} and our assumption
that $M(X)$ lies in $DM_{\cal X}$.
\end{proof}
\begin{example}\rm\llabel{nonrestr}
The unit object $\zz_{\cal X}=M({\cal X})$ of $DM_{\cal X}$ 
is usually not restricted. Consider for example the case when ${\cal
X}=\check{C}(Spec(E))$ where $E$ is a Galois extension of $k$ with the
Galois group $G$. Then 
$$Hom(\zz[i],M({\cal X}))=H_i(G,\zz)$$  
and this group may be non zero for $i>0$. If $M({\cal X})$ were
restricted this group would be equal to 
$$Hom(M({\cal X})[i],M({\cal X}))=Hom(M({\cal
X})[i],\zz)=H^{-i,0}({\cal X},\zz)$$
which is zero for $i>0$.
\end{example}
\begin{lemma}
\llabel{intres}
Let $M,N$ be objects of $DM_{\cal X}$, $P$ an object of
$DM^{eff}_{-}(S)$ and 
$$e_{\cal X}:M\oo N\sr P\oo M({\cal X})$$
a morphism such that $(N,e_{\cal X})$ is an internal Hom-object from
$M$ to $P\oo M({\cal X})$ in $DM_{\cal X}$. Assume further that $N$
is restricted. Then one has:
\begin{enumerate}
\item $(N,e_{\cal X})$ is an internal Hom-object from
$M$ to $P\oo M({\cal X})$ in $DM^{eff}_{-}(S)$
\item if $e$ is the composition
$$M\oo N\sr P\oo M({\cal X})\sr P$$
then $(N,e)$ is an internal Hom-object from
$M$ to $P$ in $DM^{eff}_{-}(S)$.
\end{enumerate}
\end{lemma}
\begin{proof}
To prove the first statement we have to show that the map
\begin{eq}
\llabel{eqmap1}
Hom(K,N)\sr Hom(K\oo M, P\oo M({\cal X}))
\end{eq}
is a bijection for any $K$ in $DM^{eff}_{-}(S)$. Since $N$ is
restricted and $M$ is in $DM_{\cal X}$ this map is isomorphic to the
map
$$Hom(K\oo M({\cal X}),N)\sr Hom(K\oo M\oo M({\cal X}), P\oo M({\cal
X}))$$
which is a bijection since $N$ is an internal Hom-object in $DM_{\cal
X}$. 

To prove the second part we have to show that the composition of
(\ref{eqmap1}) with the map
\begin{eq}
\llabel{eqmap2}
Hom(K\oo M, P\oo M({\cal X}))\sr Hom(K\oo M, P)
\end{eq}
is a bijection. This follows from the first part and the fact that
(\ref{eqmap2}) is a bijection by Lemma \ref{rl3}.
\end{proof}
%
%
\begin{lemma}
\llabel{donotneed2} Let $X$ be a smooth scheme over $S$. Then $X$ lies
in $C_{\cal X}$ if and only if the canonical morphism $u:M(X)\sr \zz$
factors through the canonical morphism $v:M({\cal X})\sr\zz$.
\end{lemma}
\begin{proof}
If $M(X)$ is in $DM_{\cal X}$ then  (\ref{mproj}) is an isomorphism
which immediately implies that $u$ factors through $v$. On the other
hand if $u=v\circ w$ where $w$ is a morphism $M(X)\sr M({\cal X})$
then 
$$(Id\oo w)\circ \Delta:M(X)\sr M(X)\oo M(X)\sr M(X)\oo M({\cal X})$$
is a section of the projection (\ref{mproj}). Therefore, $M(X)$ is a
direct summand of an object of $DM_{\cal X}$ and since $DM_{\cal X}$
is closed under direct summands we conclude that $M(X)$ is in
$DM_{\cal X}$.
\end{proof}
\begin{example}\rm\llabel{donotneed3}
If ${\cal X}=\check{C}(X)$ and $Y$ is any smooth scheme such that
$$Hom(Y,X)\ne \emptyset$$
then Lemma \ref{donotneed2} shows that $M(Y)$ lies in $DM_{\cal
X}$. In particular $M(X)$ lies in $DM_{\cal X}$. More generally for
any $\cal X$ over a perfect field one can deduce from Lemma
\ref{donotneed2} that $M(Y)$ lies in $DM_{\cal X}$ if and only if for
any point $y$ of $Y$ there exists a morphism $Spec({\cal
O}_{Y,y}^h)\sr \zz_{tr}({\cal X}_0)$ such that the composition
$$Spec({\cal
O}_{Y,y}^h)\sr \zz_{tr}({\cal X}_0)\sr \zz$$
equals $1$.
\end{example}
For an embedded $\cal X$ let 
$$\check{\cal X}=\check{C}({\cal X}_0)$$
where ${\cal X}_{0}$ is the zero term of $\cal X$.
\begin{lemma}
\llabel{descr}
There is an isomorphism $M({\cal X})\sr M(\check{\cal X})$.
\end{lemma}
\begin{proof}
Let us show that both projections
$$M({\cal X})\oo M(\check{\cal X})\sr M({\cal X})$$
and
$$M({\cal X})\oo M(\check{\cal X})\sr M(\check{\cal X})$$
are isomorphisms. Since both $\cal X$ and $\check{\cal X}$ are embedded it
is sufficient by Lemma \ref{donotneed} to show that one has
$$M({\cal X})\in DM_{\check{\cal X}}$$
and
$$M(\check{\cal X})\in DM_{{\cal X}}$$
The terms of $\cal X$ are smooth schemes ${\cal X}_i$ and for each $i$
we have
$$Hom({\cal X}_i,{\cal X}_0)\ne \emptyset$$
We conclude by Example \ref{donotneed3} that $M({\cal X}_i)$ are in
$DM_{\check{\cal X}}$ and therefore $M({\cal X})$ is in
$DM_{\check{\cal X}}$. To see the second inclusion it is sufficient to
show that $M({\cal X}_0)$ is in $DM({\cal X})$. This follows from
Lemma \ref{donotneed2} since the morphism ${\cal X}_0\sr S$ factors
through the morphism ${\cal X}\sr S$ in the obvious way.
\end{proof}
\begin{remark}\rm
Let $S=Spec(k)$ where $k$ is a field. Recall from \cite{MCnew} that
for $X$ such that $X(k)\ne \emptyset$ the projection $\check{C}(X)\sr
Spec(k)$ is a local equivalence. Since a non-empty smooth scheme over a
field always has a point over a finite separable extension of this
field we conclude from Lemma \ref{descr} that for any embedded $\cal
X$ such that $M({\cal X})\ne 0$ there exists a finite separable field
extension $E/k$ such that the pull-back of $M({\cal X})\sr \zz$ to $E$
is an isomorphism.
\end{remark}
\begin{remark}\rm
\llabel{charzero}
If we consider motives with coefficients in $R$ where $R$ is of
characteristic zero then the pull-back with respect to a finite
separable field is a conservative functor (i.e. it reflects
isomorphisms). Therefore the previous remark implies that for
$S=Spec(k)$ and motives with coefficients in a ring $R$ of
characteristic zero one has
$$M({\cal X})\cong R$$
for any non-empty embedded simplicial scheme $\cal X$. This means that
in the case of motives over a field the theory of this section is
inetresting only if we consider torsion effects.
\end{remark}
\begin{lemma}
\llabel{wheneq}
Let $\cal X$ be an embedded simplicial scheme and $X$ a smooth scheme
over $S$. Assume that the following conditions hold:
\begin{enumerate}
\item $M(X)\in DM_{\cal X}$
\item for any $Y$ such that $M(Y)\in
DM_{\cal X}$ there exists a Nisnevich covering $U\sr X$ of $X$ and a
morphism $M(U)\sr M(Y)$ such that the square
$$
\begin{CD}
M(U) @>>> M(Y)\\
@VVV @VVV\\
\zz @>Id>> \zz
\end{CD}
$$
commutes.
\end{enumerate}
Then $M(\check{C}(X))\cong M({\cal X})$.
\end{lemma}
\begin{proof}
We need to verify that the projections
$$M({\cal X}\times \check{C}(X))\sr M({\cal X})$$
and
$$M({\cal X}\times \check{C}(X))\sr M(\check{C}(X))$$
are isomorphisms. The second one is an isomorphism by Lemma
\ref{donotneed} since $M(X)\in DM_{\cal X}$ and therefore
$M(\check{C}(X))\in DM_{\cal X}$. To check that the first projection
defines an isomorphism it is sufficient by the same lemma to verify
that $M({\cal X})$ is in $DM_{\check{C}(X)}$. In view of Lemma
\ref{descr} it is sufficient to check that $M({\cal X}_0)$ is in
$DM_{\check{C}(X)}$. Since ${\cal X}_0$ is a disjoint union of smooth
varieties of finite type $Y$ such that $M(Y)$ is in $DM_{\cal X}$ it
remains to check that for such $Y$ one has 
\begin{eq}
\llabel{yisin}
M(Y)\in
DM_{\check{C}(X)}.
\end{eq}
One can easily see (cf. \cite{MCnew}) that for a
Nisnevich covering $U\sr Y$ the corresponding map $\check{C}(U)\sr
\check{C}(Y)$ is a local equivalence. Hence we may assume that $U=Y$
i.e. that there is a morphism $M(Y)\sr M(X)$ over $\zz$. Then the
morphism $M(Y)\sr \zz$ factors throught $M(\check{C}(X))$ and we
conclude by Lemma \ref{donotneed2} that (\ref{yisin}) holds.
\end{proof}
\begin{remark}\rm
One can shows that (at least over a perfect field) the conditions of
Lemma \ref{wheneq} are in fact equivalent to the condition that
$M(\check{C}(X))\cong M({\cal X})$.
\end{remark}
\begin{example}\rm\llabel{inthomex}
In the notations of Example \ref{nonrestr} consider the pair
$(M({\cal X}),e)$ where $e$ is the canonical morphism
$$M({\cal X})\oo M({\cal X})\sr M({\cal X})$$
Since $M({\cal X})$ is the unit of $DM_{\cal X}$ this pair is an
internal Hom-object from $M({\cal X})$ to itself in $DM_{\cal
X}$. However it is not an internal Hom-object from $M({\cal X})$ to
itself in $DM^{eff}_{-}(k)$ since if it were we would have
$$Hom(M, M({\cal X}))=Hom(M\oo M({\cal X}), M({\cal X}))$$
for all $M$ in $DM^{eff}_{-}(k)$ and we know that this equality does
not hold for $M=\zz$.
\end{example}

\begin{example}\rm
Since for the terms ${\cal X}_i$ of $\cal X$ we have 
$$M({\cal X}_i)\in DM_{\cal X}$$
the motive $M({\cal X})$ lies in the localizing subcategory generated
by motives of ${\cal X}_i$. If all ${\cal X}_i$ are smooth projective
varieties this implies by Lemma \ref{spcase} that $M({\cal X})$ lies
in the localizing subcategory generated by restricted
motives. Together with the previous example this shows that the
category of restricted motives is not localizing. Indeed, one can
easily see that it is closed under triangles and direct summands but
not necessarily under infinite direct sums.
\end{example}

\subsection{Coefficients}
All the results of Sections \ref{ss1}-\ref{ss5} can be immediately
reformulated in the $R$-linear context where $R$ is any commutative
ring with unit. Note that the notion of embedded simplicial scheme
depends on the choice of coefficients.

\subsection{Appendix: Internal Hom-objects}\llabel{appendix1}
Recall that for two objects $X, S$ in a tensor category an internal
Hom-object from $X$ to $S$ is a pair $(X',e)$ where $X'$ is an object
and $e:X'\oo X\sr S$ a morphism such that for any $Q$ the map
\begin{eq}
\llabel{dual}
Hom(Q,X')\sr Hom(Q\oo X,S)
\end{eq}
given by $f\mapsto e\circ (f\oo Id_{X})$ is a bijection. 

If $(X',e_X)$ is an internal Hom-object from $X$ to $S$ and $(Y',e_Y)$
an internal Hom-object from $Y$ to $S$ and we have a morphism $f:X\sr
Y$ then the composition
$$Y'\oo X\sr Y'\oo Y\stackrel{e_Y}{\sr} S$$
is the image under (\ref{dual}) of a unque morphism $Y'\sr X'$ which
we denote by $D_{S,e_X,e_Y}f$ or simply $Df$ if $S$, $e_X$ and $e_Y$
are clear from the context. One verifies easily that $D(gf)=DfDg$ if
all the required morphisms are defined. The same is true with respect
to the functoriality of internal Hom-objects in $S$.

The internal Hom-objects are unique up to a canonical isomorphism in
the following sense.
\begin{lemma}
\llabel{uniqueint}
Let $(X',e')$, $(X'',e'')$ be internal Hom-objects from $X$ to
$S$. Then there is a unique isomorphism $\phi:X'\sr X''$ such that
$e'=e''\circ (\phi\oo Id_X)$. 

If $(Y',e_Y)$ is an internal Hom-object from $Y$ to $S$ and $f:Y\sr X$
a morphism then 
$$(D'f:X'\sr Y')=(X'\stackrel{\phi}{\sr}X''\stackrel{D''f}{\sr}Y')$$
where $D'$ is the dual with respect to $e'$ and $e_Y$ and $D''$ the
dual with respect to $e''$ and $e_Y$. A similar property holds for
morphisms $X\sr Y$ and for morphisms in $S$.
\end{lemma}
A specification of internal Hom-objects for a tensor category $D$ is a
choice for each pair $(X,S)$ such that there exists an internal
Hom-object from $X$ to $S$ of one such internal Hom-object. We will
always assume below that a specification of internal Hom-objects is
fixed. The distinguished internal Hom-object from $X$ to $S$ with
respect to this specification will be denoted by
$(\uu{Hom}(X,S),ev_{X,S})$.

The construction of $D_Sf$ described above shows that for each $S$, 
\begin{eq}
\llabel{uuhom}
X\mapsto \uu{Hom}(X,S)
\end{eq}
is a contravariant functor from the full subcategory of $D$ which
consists of $X$ such that $\uu{Hom}(X,S)$ exists to $D$. Lemma
\ref{uniqueint} shows that different choices of specifications of
internal Hom-objects lead to isomorphic functors of the form
(\ref{uuhom}). The same holds for the functoriality in $S$.

Consider now the case of a tensor triangulated category $D$ which
satisfies the obvious axioms connecting the tensor and the
triangulated structure. 
We want to investigate how internal Hom-objects behave with respect to
the shift functor and distinguished triangles.

Let $X$, $S$ be a pair of objects of $D$ and $(X', e:X'\oo X\sr S)$ an
internal Hom-object from $X$ to $S$. Consider the pair $(X'[-1],
X'[-1]\oo X[1]\sr S)$ where the morphism is the composition
\begin{eq}\llabel{pair}
X'[-1]\oo X[1]\sr X'[-1][1]\oo X\sr X'\oo X\sr S
\end{eq}
One verifies easily that this pair is an internal Hom-object from
$X[1]$ to $S$. Similar behavior exists with respect to shifts of
$S$. Together with Lemma \ref{uniqueint} this shows that for a given
specification of internal Hom-objects there are canonical isomorphisms
\begin{eq}
\llabel{uuhom1}
\uu{Hom}(X[1],S)\sr \uu{Hom}(X,S)[-1]
\end{eq}
\begin{eq}
\llabel{uuhom2}
\uu{Hom}(X,S[1])\sr \uu{Hom}(X,S)[1]
\end{eq}
\begin{remark}\rm
There is another possibility for the pairing (\ref{pair}) which one
gets by moving $[-1]$ to $X$ instead of $[1]$ to $X'$. It differs from
(\ref{pair}) by sign and also makes $X'[1]$ into an internal Hom-object
from $X[1]$ to $S$. The isomorphisms (\ref{uuhom1}), (\ref{uuhom2})
constructed using different pairings (\ref{pair}) will differ by sign. 
\end{remark}
If $h:Z\sr X[1]$ is a morphism and $\uu{Hom}(Z,S)$ and $\uu{Hom}(X,S)$
exist then the composition of $Dh$ with (\ref{uuhom1}) gives a
morphism $\uu{Hom}(X,S)[-1]\sr \uu{Hom}(Z,S)$ which we will also
denote by $Dh$. This does not lead to any problems since it is always
possible to choose a specification of internal Hom-object such that
the morphisms (\ref{uuhom1}) and (\ref{uuhom2}) are identities. 

Let now
\begin{eq}\llabel{dt}
X\stackrel{f}{\sr}Y\stackrel{g}{\sr}Z\stackrel{h}{\sr}X[1]
\end{eq}
be a distinguished triangle and assume that $\uu{Hom}(X,S)$ and
$\uu{Hom}(Z,S)$ exist.
\begin{theorem}
\llabel{appmain} If $D$ satisfyies May's axiom TC3 (see \cite{MayTT}) then
for any distinguished triangle of the form
$$\uu{Hom}(Z,S)\stackrel{g'}{\sr} Y'\stackrel{f'}{\sr}
\uu{Hom}(X,S)\stackrel{Dh[1]}{\sr} \uu{Hom}(Z,S)[1]$$
there exists a morphism $e_Y:Y'\oo Y\sr S$ such that $(Y',e_Y)$ is an
internal Hom-object from $Y$ to $S$ and one has $g'=Dg$, $f'=Df$.
\end{theorem}
\begin{proof}
To simplify the notations set
$$X'=\uu{Hom}(X,S)\,\,\,\,\,e_X=ev_{X,S}$$
$$Z'=\uu{Hom}(Z,S)\,\,\,\,\,e_Z=ev_{Z,S}$$
We want to find $e_Y$ such that for any $Q$ the map
\begin{eq}
\llabel{dual2}
Hom(Q,Y')\sr Hom(Q\oo Y,S)
\end{eq}
given by $f\mapsto e_Y\circ (f\oo Id_{Y})$ is a bijection. Consider
the diagram
$$
\begin{CD}
Hom(Q,Z') @>>> Hom(Q,Y') @>>> Hom(Q,X')\\
@VVV @. @VVV\\
Hom(Q\oo Z,S) @>>> Hom(Q\oo Y, S) @>>> Hom(Q\oo X, S).
\end{CD}
$$
If we can find $e_Y$ such that the corresponding map (\ref{dual2})
subdivides this diagram into two commutative squares then this map
will be a bijection by the Five Lemma. In addition setting $Q=Z'$ and
using the commuuativity of the left squares on $Id_{Z'}$ we get
$g'=Dg$ and setting $Q=Y'$ and using the commutativity of the right
square on $Id_{Y'}$ we get $f'=Df$. It is sufficient therefore to find
$e_Y$ which satisfy the two commutativity conditions.

A simple diagram chase shows that the commutativity of the left square
is equivalent to the commutativity of the square
$$
\begin{CD}
Z'\oo Y @>>> Y'\oo Y\\
@VVV @VVe_YV\\
Z'\oo Z @>e_Z>> S
\end{CD}
$$
and the commutativity of the right square to the commutativity of the
square
$$
\begin{CD}
Y'\oo X @>>> Y'\oo Y\\
@VVV @VVe_YV\\
X'\oo X @>e_X>> S
\end{CD}
$$
Together we may express our condition as the commutativity of the
square
$$
\begin{CD}
(Y'\oo X)\oplus (Z'\oo Y) @>>> Y'\oo Y\\
@VVV @VVe_YV\\
(X'\oo X)\oplus (Z'\oo Z) @>e_X+ e_Z>> S
\end{CD}
$$
Applying Axiom TC3' (\cite{MayTT}) to our triangles we see that there is an
object $W$ which fits into a commutative diagram
$$
\begin{CD}
(Y'\oo X)\oplus (Z'\oo Y) @>>> Y'\oo Y\\
@VVV @VVk_2V\\
(X'\oo X)\oplus (Z'\oo Z) @>k_3+k_1>> W
\end{CD}
$$
It remains to show that $e_X+ e_Z$ factors through $k_3+ k_1$. By
\cite[Lemma 4.9]{MayTT} the lower side of this square extends to an
exact triangle of the form
$$(X'\oo Z)[-1]\sr (X'\oo X)\oplus (Z'\oo Z) \stackrel{k_3+k_1}{\sr} W\sr X'\oo
Z$$
Therefore it is sufficient to show that the diagram
$$
\begin{CD}
(X'\oo Z)[-1] @>>> X'\oo X`\\
@VVV @VVe_XV\\
Z'\oo Z @>e_Z>> S
\end{CD}
$$
anticommute. A diagram of this form can be defined for any morphism of
the form $X\sr Z[1]$ and its anticommutativity follows easily from the
elementary axioms.
\end{proof}
\begin{remark}\rm
\llabel{appmain'}
Applying Theorem \ref{appmain} to the opposite category to $D$ one
concludes that a similar result holds for distinguished triangles in
$S$.
\end{remark} 
Theorem \ref{appmain} together with the preceeding discussion of
internal Hom-objects and the shift functor, implies in particular that
for a given $S$ (resp. given $X$) the subcategory $D_{(-,S)}$
(resp. $D_{(X,-)}$) which consists of all $X$ (resp. all $S$) such
that $\uu{Hom}(X,S)$ exists is a triangulated subcategory.
\begin{proposition}
\llabel{homex}
The functors
$$\uu{Hom}(-,S):D_{(-,S)}\sr D$$
$$\uu{Hom}(X,-):D_{(X,-)}\sr D$$
considered together with the canonical isomorphisms (\ref{uuhom1}),
(\ref{uuhom2}) are triangulated functors.
\end{proposition}
\begin{proof}
It is clearly sufficient to prove the part of the proposition related
to $\uu{Hom}(-,S)$ i.e. to show that this functor takes distinguished
triangles to distinguished triangles. Consider a distinguished
triangle of the form (\ref{dt}) and the resulting triangle
\begin{eq}
\llabel{needd}
\uu{Hom}(Z,S)\stackrel{Dg}{\sr}\uu{Hom}(Y,S)\stackrel{Df}{\sr}\uu{Hom}(X,S)\stackrel{Dh[1]}{\sr}\uu{Hom}(Z,S)[1]
\end{eq}
In view of Theorem \ref{appmain} there exists an internal Hom-object
$(\tilde{Y}',\tilde{e}_Y)$ from $Y$ to $S$ such that the triangle
formed by $\tilde{D}g$, $\tilde{D}f$ and $Dh[1]$ is distinguished. By
Lemma \ref{uniqueint} there is an isomorphism $\tilde{Y}'\sr
\uu{Hom}(Y,S)$ which extends to an isomorphism of triangles. We
conclude that (\ref{needd}) is isomorphic to a distinguished triangle
and therefore is distinguished.
\end{proof}

\begin{thebibliography}{1}

\bibitem{MayTT}
J.~P. May.
\newblock The additivity of traces in triangulated categories.
\newblock {\em Adv. Math.}, 163(1):34--73, 2001.

\bibitem{Nee}
Amnon Neeman.
\newblock Triangulated categories.
\newblock 148, 2001.

\bibitem{Delnotes}
Vladimir Voevodsky.
\newblock Lectures on motivic cohomology 2000/2001 (written by {P}ierre
  {D}eligne).
\newblock {\em www.math.uiuc.edu/K-theory/527}, 2000/2001.

\bibitem{MCnew}
Vladimir Voevodsky.
\newblock On 2-torsion in motivic cohomology.
\newblock {\em www.math.uiuc.edu/K-theory/502}, 2001.

\bibitem{cancellation}
Vladimir Voevodsky.
\newblock Cancellation theorem.
\newblock {\em www.math.uiuc.edu/K-theory/541}, 2002.

\bibitem{zslice}
Vladimir Voevodsky.
\newblock On the zero slice of the sphere spectrum.
\newblock {\em www.math.uiuc.edu/K-theory/612}, 2002.

\bibitem{zl}
Vladimir Voevodsky.
\newblock Motivic cohomology with ${\bf z}/l$-coefficients.
\newblock {\em In preparation}, 2003.

\bibitem{collection}
Vladimir Voevodsky, Eric~M. Friedlander, and Andrei Suslin.
\newblock {\em Cycles, transfers and motivic homology theories}.
\newblock Princeton University Press, 2000.

\end{thebibliography}
\end{document}